\documentclass[12pt]{amsart}

\usepackage[latin1]{inputenc}
\usepackage[english]{babel}
\usepackage{amsmath,amssymb,amsthm} 
\usepackage{enumitem}
\usepackage{color}
\usepackage{hyperref}

\definecolor{dark_purple}{rgb}{0.4, 0.0, 0.4}
\definecolor{dark_green}{rgb}{0.1, 0.6, 0.2}

\newtheorem{theorem}{Theorem}[section]
\newtheorem{lemma}[theorem]{Lemma}

\theoremstyle{definition}
\newtheorem{definition}[theorem]{Definition}

\newtheorem{conjecture}[theorem]{Conjecture}
\newtheorem*{remark*}{Remark}

\renewcommand{\Re}{\operatorname{Re}}
\newcommand{\A}{{\mathcal A}}
\newcommand{\B}{{\mathcal B}}
\renewcommand{\H}{{\mathcal H}}
\newcommand{\K}{{\mathcal K}}
\newcommand\C{{\mathbb C}}
\newcommand\matn{{\mathcal M}_n}
\newcommand\ta{{\tilde{A}}}

\title{From Clou\^atre-Ostermann-Ransford to Okubo-Ando}
\author[M. Hartz]{Michael Hartz}
\address{Fachrichtung Mathematik, Universit\"at des Saarlandes, 66123 Saarbr\"ucken, Germany}
\thanks{M.H. was partially supported by the Emmy Noether Program of the German Research Foundation (DFG Grant 466012782)}
\email{hartz@math.uni-sb.de}
\author[J. E. M\textsuperscript{c}Carthy]{John E.\ M\textsuperscript{c}Carthy}
\thanks{J.M. partially supported by National Science Foundation Grant  
DMS 2554487}
\address{Dept. of Mathematics, Washington University, St. Louis MO 63130, USA}
\email{mccarthy@wustl.edu}
\date{\today}


\begin{document}
\begin{abstract}
We prove that if $\theta$ is a continuous unital homomorphism of an operator algebra $\A$ into $B(\mathcal{H})$, and $\beta$ is in the dual space of $\A$, then the completely bounded norm of $\theta$ is less than or equal to the maximum of $1$ and the completely bounded norm of $\theta + \beta I $. As an application, we give another proof of the Okubo--Ando theorem.
\end{abstract}

\maketitle

\section{Introduction}
The Crouzeix conjecture is that the numerical range of an operator on a Hilbert space is always a 2-spectral set for the operator \cite{Crouzeix07}.
The sharpest result proved to date is by Crouzeix and Palencia \cite{CP17}, who proved that the numerical range is a $(1 + \sqrt{2})$-spectral set.
Examining the proofs in \cite{CP17,RS17}, Clou\^atre, Ostermann and Ransford were led to make the following conjecture, which if true would imply the Crouzeix conjecture \cite{COR23}.
\begin{conjecture} \cite{COR23}
Let $\A$ be a uniform algebra, and $\alpha: \A \to \A$ a unital anti-linear contraction.
Let $\theta: \A \to \matn(\C)$ be a continuous unital algebra homomorphism.

(i) 
If the linear map 
\[
\Lambda: \mathcal{A} \to M_n(\mathbb{C}), \quad f \mapsto \theta(f) + \theta (\alpha(f))^*,
\]
has norm at most $2$, then $\| \theta \| \leq 2$.

(ii) If $\alpha$ is a complete contraction, $\theta$ is completely bounded, and $\| \Lambda \|_{cb} \le 2$, then 
$\| \theta \|_{cb} \le 2$.
\end{conjecture}
They proved that  both parts of their conjecture hold with the weaker conclusion $\| \theta \| \le 1 + \sqrt{2} $ (resp. $\| \theta \|_{cb} \le 1 + \sqrt{2}$). They also proved that (i) 
  holds with constant $2$ if the range of $\alpha $ is the scalar multiples of the identity, even without the assumption that $\alpha(1) = 1$.
  They asked if (ii) also holds under this hypothesis. The object of this note is to prove that it does,
  which is an immediate consequence of the following result.
\begin{theorem}
  \label{thm:main_OA}
  Let $\mathcal{A}$ be a unital operator algebra and let $\theta: \mathcal{A} \to B(\mathcal{H})$ be a continuous unital
  homomorphism. Then for all $n \in \mathbb{N}$ and all $\beta$ in the dual space of $\mathcal{A}$,
  \begin{equation*}
    \|\theta^{(n)}\| \le \max(1,  \|\theta^{(n)} + \beta^{(n)} I \|).
  \end{equation*}
\end{theorem}
Here, $(\theta + \beta I)(a) = \theta(a) + \beta(a) I$, and the notation $\theta^{(n)}$ means the $n$-th matrix ampliation; see Definition \ref{defbc}.
The case $n=1$ was proved in \cite[Thm. 4.1]{COR23}. (They stated the theorem for uniform algebras, but their proof extends to all operator algebras). 
The chief reason to consider ampliations of $\theta$ is that control of the completely bounded norm $\|\theta\|_{cb} = \sup_{n} \|\theta^{(n)}\|$
is crucial from the point of view of dilation theory. We discuss one such application.

The class $C_\rho$ of operators with a unitary $\rho$-dilation (see Definition \ref{def_crho}) was introduced by Sz.-Nagy and Foia\c{s}.
Operators of class $C_\rho$ can be characterized intrinsically; see \cite[Section I.11]{SFB+10}.
In particular, the class $C_1$ is the class of contractions (operators of norm at most one); this is the Sz-Nagy dilation theorem
\cite{Sz.-Nagy53}.
 $C_{2}$ consists of all operators whose numerical range lies in the closed unit disk; this is Berger's strange dilation theorem \cite{Berger65}.

Every operator in $C_\rho$ was shown to be  similar to a contraction in \cite{SF67}. Okubo and Ando obtained the sharp bound on the condition number \cite{OA75}.
\begin{theorem}[Okubo--Ando]
  \label{thm:OA}
  Let $\rho \ge 1$ and let $T \in B(\mathcal{H})$ be of class $C_{\rho}$. Then there exists an invertible
  operator $S \in B(\mathcal{H})$ such that $\|S^{-1} T S\| \le 1$ and such that $\|S\| \|S^{-1}\| \le \rho$.
\end{theorem}

We give another proof of the Okubo--Ando theorem in Section \ref{secoa}.

\section{Definitions}

We recall a few definitions.
By an operator algebra, we mean  a  subalgebra $\mathcal{A}$ of $B(\mathcal{K})$
for some Hilbert space $\mathcal{K}$; we do not assume that the algebra is self-adjoint. Matrices with entries from $\A$ come  equipped with natural norms. Indeed, if 
 $\A \subset B(\K)$, then $\matn(\A)$, the $n$-by-$n$ matrices with entries from $\A$,
is a subalgebra of $\matn(B(\mathcal{K}))$, and this can be canonically identified with 
$ B(\mathcal{K}^n)$. The norm on $\matn(\A)$ is its norm as a subalgebra of  $ B(\mathcal{K}^n)$.

\begin{definition}
\label{defbc}
Let $\A $ and $\B $ be operator algebras, and $\theta: \A \to \B$ a mapping.
We write $\theta^{(n)}: \matn(\mathcal{A}) \to \matn(\B)$ for the $n$-th ampliation of the map $\theta$, defined by applying $\theta$ entry-wise.
We define \[
\| \theta\|_{cb} = \sup_{n \in {\mathbb N}} \| \theta^{(n)}\|.\] We say $\theta$ is completely bounded if
$\| \theta \|_{cb} < \infty$, and say $\theta$ is a complete contraction if
$\| \theta \|_{cb} \le 1$.
\end{definition}
Since the work of Stinespring \cite{Stinespring55} and Arveson \cite{Arveson69}, it has been known that studying the ampliations can be very helpful in understanding the original map.
\begin{definition}
\label{def_crho}
Let $\rho > 0$. An operator $T \in B(\mathcal{H})$ is said to be of class $C_\rho$ if there exists a unitary operator $U$ on a Hilbert space containing $\mathcal{H}$ such that
\begin{equation*}
  T^n = \rho P_{\mathcal{H}} U^n \big|_{\mathcal{H}} \quad \text{ for all } n \ge 1.
\end{equation*}
In this case, $U$ is called a unitary $\rho$-dilation of $T$.
Equivalently,
\begin{equation}
  \label{eqn:rho_dilation}
  f(T) + (\rho-1) f(0) = \rho P_{\mathcal{H}} f(U) \big|_{\mathcal{H}} \quad \text{ for all } f \in \mathbb{C}[z].
\end{equation}
\end{definition}

\section{Proof of the Okubo-Ando theorem}
\label{secoa}

The conclusion of the Okubo--Ando theorem (Theorem \ref{thm:OA}) implies
that if $T$ is of class $C_\rho$, then the inequality
\begin{equation}
  \label{eqn:OA_ineq}
  \|f(T)\| \le \rho \|f\|_{\overline{\mathbb{D}}}
\end{equation}
holds for all $f \in \mathbb{C}[z]$, an improvement from the obvious bound
\[
\|f(T)\| \le \rho  \|f\|_{\overline{\mathbb{D}}} + |\rho -1| |f(0)|
\]
from \eqref{eqn:rho_dilation}.

The original proof  of Okubo and Ando \cite{OA75} is based on analyzing geometric properties of the unitary $\rho$-dilation of $T$.
More recently, Caldwell, Greenbaum and Li \cite{CGL17} gave a simple direct proof of Inequality \eqref{eqn:OA_ineq} for $\rho = 2$
(in the case of finite dimensional Hilbert spaces). Their proof is based on a variational argument, which they attribute to Crouzeix.
This was extended to a proof of the general Inequality \eqref{eqn:OA_ineq} by Clou\^atre, Ostermann and Ransford \cite{COR23}.

On the other hand,
Theorem \ref{thm:main_OA} shows that Inequality \eqref{eqn:OA_ineq} also holds for matrix-valued polynomials,
so that Theorem \ref{thm:OA} follows from Paulsen's similarity theorem, which says that the completely bounded norm of the map 
\[
\mathbb{C}[z]\to B(\H),\ f \mapsto f(T)
\]
equals the infimum of the condition numbers $\|S\| \|S^{-1}\|$ over all $S$ such that $\|S^{-1} T S\| \le 1$.

\begin{proof}[Proof of Theorem \ref{thm:OA}]
  Let $\mathcal{A} = \mathbb{C}[z]$, equipped with the supremum norm over $\overline{\mathbb{D}}$. Define 
  \[
    \beta: \mathbb C[z] \to \mathbb{C}, \quad \beta(f) = (\rho-1) f(0),
  \]
  and let $\theta$ be the polynomial functional calculus of $T$. Then \eqref{eqn:rho_dilation} shows that $\|\theta + \beta I \|_{cb} \le \rho$, hence $\|\theta\|_{cb} \le \rho$ by Theorem \ref{thm:main_OA}. Paulsen's similarity theorem \cite[Theorem 9.1]{Paulsen02} then yields the invertible operator $S$.
\end{proof}

\section{Proof of main result}

The key to proving Theorem \ref{thm:main_OA} is to adapt the variational argument in \cite{CGL17} to matrix-valued polynomials. To this end, we will use biholomorphic
automorphisms of the unit ball of $B(\H)$.
In general, if $\mathcal{H}$ is a Hilbert space and if $A \in B(\mathcal{H})$ with $\|A\| < 1$, the Potapov--M\"obius transform on the unit ball of $B(\mathcal{H})$
is defined by
\begin{equation}
\label{eqpm}
  m_A(X) = (I - A A^*)^{-1/2} (X + A)(I + A^* X)^{-1} (I - A^* A)^{1/2}.
\end{equation}
If $\|A\| < 1$ and $\|X\| \le 1$, then $\|m_A(X)\| \le 1$, i.e.\ $m_A$ maps the closed unit ball of $B(\mathcal{H})$ into itself; see e.g.\ \cite[Theorem 8.19]{IS85}.
Let us note that even if $\| X \| > 1$, formula \eqref{eqpm} makes sense as long as  $\|A\| \|X\| < 1$.

Of particular importance for us is the case when the Hilbert space is of the form $\mathcal{H}^n$, so $B(\mathcal{H}^n) = \matn(B(\mathcal{H})) = \matn(\mathbb{C}) \otimes B(\mathcal{H})$. Given a scalar matrix $A \in \matn(\mathbb{C})$ with $\|A\| < 1$, we
shall define
\[
\ta = A \otimes I_{\mathcal{H}}
\]
 and thus define $m_\ta(T)$ for $T \in \matn(B(\mathcal{H}))$ with $\|A\| \|T\| < 1$.
In particular, we have $\|m_\ta(T)\| \le 1$ whenever $\|T\| \le 1$. We will also
use the observation that if $T \in M_n(\mathcal{A})$ for some norm-closed unital operator algebra $\mathcal{A}$,
then $m_{\ta}(T) \in M_n(\mathcal{A})$ for all $A \in M_n(\mathbb{C})$ with $\|A\| < 1$ and $\|A\| \|T\| < 1$.

The following lemma is a generalization of \cite[Theorem 5.1]{CGL17} to matrix-valued functions, sequences of operators and infinite dimensional Hilbert spaces.

\begin{lemma}
  \label{lem:extremal}
  Let $(R_k)$ be a sequence in $\matn(B(\mathcal{H}))$. Let $C \in (0,\infty) \setminus \{1\}$.
  Assume that whenever $A \in \matn(\mathbb{C})$ satisfies $\|A\| \|R_k\| < 1$ then $ \|m_\ta(R_k)\| \le C$.

  Let $(x_k)$ be a sequence of unit vectors in $\mathcal{H}^n$ with $\lim_{k \to \infty} \|R_k x_k\| = C$.
  Then for every bounded sequence $(A_k)$ in $\matn(\mathbb{C})$,
  \begin{equation*}
    \lim_{k \to \infty} \langle R_k x_k, (A_k \otimes I_{\mathcal{H}}) x_k \rangle  =0.
  \end{equation*}
\end{lemma}

\begin{proof}
  The assumption implies that $\|R_k\| \le C$ for all $k$ by taking $A=0$.
  Let $A \in \matn(\mathbb{C})$ with $C \|A\| < 1$ and let $y \in \mathcal{H}^n$ be a unit vector.
  Define $z_k = (I - \ta^* \ta)^{-1/2} (I + \ta^* R_k) y$.
  Then by assumption,
  \begin{align*}
    \|(I - \ta^* \ta)^{-1/2} (R_k + \ta) y\|^2  &= \|m_\ta(R_k) z_k\|^2 \\ &\le C^2 \|z_k\|^2  \\
                                                        &= C^2 \|(I - \ta^* \ta)^{-1/2} (I + \ta^* R_k) y\|^2.
  \end{align*}

  We now replace $A$ with $t A$ for a small complex number $t$ and expand to first order.
  We have $(I - |t|^2 \ta \ta^*)^{-1/2} = I + \mathcal{O}(|t|^2)$ and $(I - |t|^2 \ta^* \ta)^{-1/2} = I + \mathcal{O}(|t|^2)$.
  Therefore, the last inequality implies that
  \begin{equation*}
    \|R_k y\|^2 + 2 \Re \overline{t} \langle R_k y, \ta y \rangle \le C^2 ( \|y\|^2 + 2 \Re t \langle y, \ta^* R_k y \rangle )
    + \mathcal{O}(|t|^2),
  \end{equation*}
  where the implied constant in the $\mathcal{O}(|t|^2)$ term is independent of the unit vector $y$ and the number $k$
  because $\|R_k\| \le C$.
  Rearranging yields
  \begin{equation*}
 2   (1 - C^2)  \Re \overline{t} \langle R_k y , \ta y \rangle \le C^2 \|y\|^2 - \|R_k y\|^2 + \mathcal{O}(|t|^2)
  \end{equation*}
  for all complex numbers $t$ in a neighborhood of zero and all unit vectors $y$.
  By varying the argument of $t$, we conclude that
  \begin{equation*}
   2 |1 - C^2|  |t| | \langle R_k y , \ta y \rangle| \le C^2 \|y\|^2 - \|R_k y\|^2 + \mathcal{O}(|t|^2).
  \end{equation*}

  We apply the last inequality to $y = x_k$, the sequence in the statement of the lemma. Since the estimate
  above is independent of the unit vector $y$ and $k$, it follows that
  \begin{equation*}
 2   |1 - C^2|  |t| \limsup_{k \to \infty} | \langle R_k x_k , \ta x_k \rangle| \le \mathcal{O}(|t|^2)
  \end{equation*}
  for all $t$ in a neighborhood of zero. Since $C \neq 1$, we conclude that
  \[
    \lim_{k \to \infty} \langle R_k x_k, \ta x_k \rangle = 0.
  \]
  We have proved this identity for all $A \in \matn(\mathbb{C})$ with $ C \|A\| < 1$,
  but then by homogeneity, it holds for all $A \in \matn(\mathbb{C})$. This shows that
  the linear functionals
  \begin{equation*}
    \matn(\mathbb{C}) \to \mathbb{C}, \quad A \mapsto \langle \ta x_k, R_k x_k \rangle,
  \end{equation*}
  converge to zero pointwise, and hence in norm since $\matn(\mathbb{C})$ is finite dimensional. This completes the proof.
\end{proof}


\begin{proof}[Proof of Theorem \ref{thm:main_OA}]
  Given Lemma \ref{lem:extremal}, the proof proceeds along the same lines as that of \cite[Theorem 4.1]{COR23}.
  Since the details differ, we provide the argument. 
Let $\A \subset B(\K)$. Since $\theta$ and $\beta$ are continuous, we may replace $\mathcal{A}$ with its
norm closure if necessary and hence assume that $\mathcal{A}$ is closed.
Fix $n \in \mathbb{N}$ and let $C = \|\theta^{(n)}\|$.
  If $C \le 1$, there is nothing to prove, so assume $C > 1$.

  Let $(a_k)$ be a sequence in the unit ball of $\matn(\mathcal{A})$
  and let $(x_k)$ be a sequence of unit vectors in $\mathcal{H}^n$ such that $\|\theta^{(n)}(a_k) x_k\| \to C$.
  Let $R_k = \theta^{(n)}(a_k) \in \matn(B(\mathcal{H}))$.
  Then $\|R_k\| \ge 1$ for sufficiently large $k$; we may assume that this holds for all $k$.
  We claim that the $R_k$ satisfy
  the assumption of Lemma \ref{lem:extremal}. Indeed, if $A \in \matn(\mathbb{C})$ with $\|A\| \|R_k\| < 1$,
  then $\|A\| < 1$, and since $\theta$ is a unital homomorphism,
  \begin{equation*}
    m_{A \otimes I_{\H}}(R_k) = m_{A \otimes I_{\H}}(\theta^{(n)}(a_k)) = \theta^{(n)} (m_{A \otimes I_{\K}}(a_k)).
  \end{equation*}
  The discussion preceding Lemma \ref{lem:extremal} shows that $m_{A \otimes I_{\K}}(a_k)$ belongs to the closed unit ball of $\matn(\mathcal{A})$,
  hence $\|m_{A \otimes I_{\H}}(R_k)\| \le C$ for all $k$, as claimed.
  Now
  \begin{align*}
    \|\theta^{(n)} + \beta^{(n)} I \|^2
    &\ge \| (\theta^{(n)} + \beta^{(n)} I) (a_k) x_k\|^2 \\
    &\ge \|\theta^{(n)} (a_k) x_k\|^2 + 2 \Re \langle R_k x_k , (\beta^{(n)}(a_k) \otimes I_{\mathcal{H}} ) x_k \rangle 
  \end{align*}
  for all $k$. Taking the limit as $k \to \infty$, the second summand vanishes by Lemma \ref{lem:extremal}, and the first summand
  converges to $C^2 = \|\theta^{(n)}\|^2$, so
  \[
    \|\theta^{(n)}\| \le \|\theta^{(n)} + \beta^{(n)} I \|. \qedhere
  \]
\end{proof}

\bibliographystyle{amsplain}
\bibliography{literature}

\end{document}